\documentclass[12pt,a4paper]{amsart}
\usepackage[pagewise]{lineno}
\usepackage[top=35mm, bottom=32mm, left=30mm, right=30mm]{geometry}
\usepackage{newtxtext,newtxmath}

\usepackage{mathrsfs}
\usepackage{fancyhdr} 
\usepackage{url}
\usepackage[all]{xy}
\usepackage[colorlinks=
true
, linkcolor=red, urlcolor=red, citecolor=red, linktoc=page,citecolor=blue]{hyperref}
\usepackage{enumitem}

\newtheorem{theorem}{Theorem}[section]
\newtheorem{proposition}[theorem]{Proposition}
\newtheorem{lemma}[theorem]{Lemma}
\newtheorem{corollary}[theorem]{Corollary}

\newtheorem{definition}[theorem]{Definition}

\theoremstyle{definition}

\newcommand{\mdim}{\operatorname{mdim}}
\newcommand{\Map}{\operatorname{Map}}
\newcommand{\Sym}{\operatorname{Sym}}
\newcommand{\Fix}{\operatorname{Fix}}

\newcommand{\ind}{\alpha}

\newcommand{\cG}{\mathcal G}
\newcommand{\cU}{\mathcal U}
\newcommand{\cV}{\mathcal V}

\newcommand{\N}{\mathbb N}
\newcommand{\Z}{\mathbb Z}

\title{A Hard-Core Subshift Whose Sofic Mean Dimension Depends on the Sofic Approximation}
\author{Xianqiang Li and Zhuowei Liu *}
\date{\today}
\address[X. Li]{Shanghai center for Mathematial Science, Fudan University, Shanghai,  200000, P.R. China}
\email{26110840010@m.fudan.edu.cn}
\address[Z. Liu]{School of Mathematics (Zhuhai), Sun Yat-sen University,
	Zhuhai, Guangdong, 519000, P.R. China}
\email{liuzhw55@mail2.sysu.edu.cn}

\thanks{* Zhuowei Liu is the corresponding author.}
\begin{document}
	\maketitle
	
	\begin{abstract}
  		We exhibit a topologically mixing continuous-alphabet subshift of the free group $F_2$ whose sofic mean dimension depends on the chosen homomorphic sofic approximation, which gives an answer of Li \cite[Remark 2.7]{Li13}.  The system is the hard-core subshift
		\[
		X_{\rm hc}=\{x\in [0,1]^{F_2}:x_gx_{gs}=0\text{ for every }g\in F_2\text{ and }s\in\{a,b\}\}.
		\]
		
		We construct one sofic approximation from finite quotients compatible with the parity homomorphism $F_2\to\Z/2\Z$; all of its action graphs are bipartite and give sofic mean dimension exactly $1/2$.  A second approximation is selected from two independent uniform random permutations and gives a value in $[1/5,9/20]$.  Consequently a mixing action can have two distinct positive sofic mean dimensions.
	\end{abstract}
	
	\section{Introduction}
	The mean dimension of a continuous map $T : X \to X$ on a compact metrizable space $X$ was introduced by Gromov~\cite{MG}. 
It was subsequently developed for action of amenable groups via F\o lner sequences by Lindenstrauss and
Weiss ~\cite{LW00} and extended to certain non-amenable groups by Li~\cite{Li13} via sofic approximations in recent years. 
It has since become a fundamental invariant in the study of infinite-dimensional dynamical systems, particularly in embedding problems and the geometric theory of topological dynamical system. 

If a sofic group $G$ acts by homeomorphisms on a compact metrizable space $X$ and a sofic approximation $\Sigma$ to $G$ is given, then the mean dimension of the action is a topological conjugacy invariant, denoted by $\mathrm{mdim}_\Sigma(X) \in \{-\infty\}\cup[0, \infty]$.
It is also called sofic mean dimension if $\Sigma$ is understood.
It was pioneered in~\cite{Li13} and coincides with the classical mean dimension when the group is amenable. Its definition depends formally on a chosen sofic approximation sequence.  The author of ~\cite{Li13} observed that no example was known in which changing the approximation changes the value (see also \cite[Remark 2.21]{GRFYG}).  More recently, Jin and Qiao proved approximation-independence for all amenable group actions and for full shifts over arbitrary compact metrizable alphabets \cite{JQ24}.  Thus any dependence phenomenon must use a genuinely nonamenable action and a subsystem more constrained than a full shift. The goal of this paper is to construct an example demonstrating that the sofic mean dimension depends on the choice of the sofic approximation sequences.
	
	For sofic entropy, Airey, Bowen and Lin \cite{ABL22} constructed a mixing subshift of finite type with two different positive sofic entropies.  Their proof relies on a comparison between a uniform random model and a planted model, showing that an exponential discrepancy in the number of proper colorings yields a corresponding difference in entropy. Our construction adopts the same strategy. However, the underlying finite statistic differs: for mean dimension, the relevant quantity is not the number of colorings but the count of continuously free coordinates. In the hard-core model, these coordinates correspond precisely to the independent vertices of the finite action graph.
	
	Let
	\[
	F_r=\langle s_1,\ldots,s_r\rangle,
	\qquad S=\{s_1,\ldots,s_r\},
	\]
	and consider
	\begin{equation}\label{eq:hardcore}
		X_{\rm hc}=\left\{x\in[0,1]^{F_r}:x_gx_{gs}=0
		\text{ for all }g\in F_r,\ s\in S\right\}.
	\end{equation}
	The group acts by the left shift
	\[
	(hx)_g=x_{h^{-1}g}.
	\]
	The symbol $0$ is safe: arbitrary legal patterns on sufficiently separated finite sets can be joined by filling all remaining coordinates with $0$.  In particular, the action is topologically mixing (see Proposition~\ref{prop:mixing}). Prior research focuses on the discrete configuration space $\{0,1\}^G$, referred to as the hard-ball model, which possesses the pseudo-orbit tracing property (see \cite[Example 6.6]{BGL22} and \cite[Example 4.1]{CCL}).
	
	Suppose that
	\[
	\Sigma=(\sigma_n:F_r\to\Sym(V_n))_{n\ge 1}
	\]
	is a sofic approximation consisting of homomorphisms.  The action graph $\cG_{\sigma_n}$ has vertex set $V_n$ and undirected edges
	\[
	\{v,\sigma_n(s_j)v\},\qquad v\in V_n,\quad 1\le j\le r.
	\]
	Loops and multiple edges are allowed. 	A set $I\subseteq V_n$ is independent if it contains no two endpoints of an edge and contains no looped vertex.  Write $\ind(\cG_{\sigma_{n}})$ for the independence number.
	
	Our first theorem computes both sofic mean dimension and sofic metric mean dimension.
	
	\begin{theorem}[Model-dimension formula]\label{thm:model-formula}
		Let $\Sigma=(\sigma_n:F_r\to\Sym(V_n))_{n\ge 1}$ be a homomorphic sofic approximation.  For the hard-core subshift \eqref{eq:hardcore}, set
		\[
		q_\Sigma=\limsup_{n\to\infty}
		\frac{\ind(\cG_{\sigma_n})}{|V_n|}.
		\]
		Let $\rho_0(x,y)=|x_e-y_e|$.  Then
		\[
		\mdim_\Sigma(X_{\rm hc})
		=\underline{\mdim}_{\Sigma,\mathrm M}(X_{\rm hc},\rho_0)
		=\overline{\mdim}_{\Sigma,\mathrm M}(X_{\rm hc},\rho_0)
		=q_\Sigma.
		\]
		Thus the lower and upper sofic metric  mean dimensions coincide for the dynamically generating continuous pseudometric $\rho_0$.
	\end{theorem}
	
	The exact lower bound is encoded by a maximum independent set $I_n$.  Every vector in $[0,1]^{I_n}$ defines an exact equivariant microstate by placing the vector on $I_n$, placing $0$ elsewhere, and taking pullback names.  This produces a cube of dimension $|I_n|$ inside every microstate space.
	
	For the upper bound, if $\varphi:V_n\to X_{\rm hc}$ is approximately equivariant and $y_v=\varphi(v)_e$, then for every threshold $\tau>0$ the set
	\[
	A_\tau(\varphi)=\{v:y_v\ge \tau\}
	\]
	contains only a controlled number of internal generator edges.  Deleting at most $O(\delta^2\tau^{-2}|V_n|)$ vertices makes it independent.  Hence only $\ind(\cG_{\sigma_n})+o(|V_n|)$ coordinates can carry values above $\tau$. Quantization of these coordinates and enumeration of the threshold set provide the necessary separated-set bound, avoiding a subtle direct covering-dimension analysis of approximate microstates. 
	
	The second theorem supplies two sofic approximation sequences.
	
	\begin{theorem}[Approximation dependence]\label{thm:main}
		There exist two homomorphic sofic approximation sequences $\Sigma_{\rm unif}$ and $\Sigma_{\rm bip}$ of $F_2$ such that
		\[
		\frac15
		\le \mdim_{\Sigma_{\rm unif}}(X_{\rm hc})
		\le \frac9{20}
		<\frac12
		=\mdim_{\Sigma_{\rm bip}}(X_{\rm hc}).
		\]
		In particular, the sofic mean dimension of a topologically mixing action can depend on the chosen sofic approximation, and both values can be positive and finite.
	\end{theorem}
	To the best of our knowledge, Theorem \ref{thm:main} gives the first example of two distinct nonnegative, indeed positive, values of sofic mean dimension for one action.
    
	The bipartite approximation is obtained from finite quotients lying in the kernel of the parity map $F_2\to\Z/2\Z$.  One parity class is an independent set, while the edges of one generator contain a perfect matching, so the independence ratio is exactly $1/2$.
	
	For the uniform approximation, take two independent uniform random permutations.  For a fixed set $I$ of density $q\le 1/2$, the probability that a uniform permutation sends $I$ disjointly from itself is
	\[
	\frac{((1-q)n)!^2}{((1-2q)n)!n!}
	\]
	when the displayed quantities are integral.  The first moment for independent sets has exponential rate
	\begin{equation}\label{eq:Phi-intro}
		\Phi_2(q)=H(q)+2\bigl(2(1-q)\log(1-q)
		-(1-2q)\log(1-2q)\bigr),
	\end{equation}
	where $H(q)=-q\log q-(1-q)\log(1-q)$.  Since $\Phi_2(9/20)<-0.16$, the probability of an independent set of density at least $9/20$ tends to zero exponentially.  At the same time, random permutation models are asymptotically free, so a deterministic diagonal choice gives a sofic approximation satisfying the independence bound.
	
    The paper is organized as follows. In Section \ref{sec2}, we recall some definitions and some related results. In Section \ref{sec3}, we show that the hard-core system is topologically mixing and prove Theorem~\ref{thm:model-formula}. In Section \ref{sec4}, we construst a bipartite sofic approximation and show that the sofic mean dimension of the hard-core system with respect to this sofic approximation is $\frac{1}{2}$. In Section \ref{sec5}, we prove Theorem~\ref{thm:main}.
    
\emph{Acknowledgments.} We would like to thank Professor Siming Tu, who provided much guidance and assistance.
	
	\section{Preliminaries}\label{sec2}
	
	\subsection{Sofic approximations}
	
	Let $G$ be a countable group with identity $e$ and $\sigma: G \to \Sym(V)$ be a map where $V$ is a finite set and $\Sym(V)$ is the group of permutations of $V$.  A \emph{sofic approximation sequence} is a sequence
	\[
	\Sigma=(\sigma_n:G\to\Sym(V_n))_{n\ge1},
	\qquad |V_n|\to\infty,
	\]
	which is asymptotically multiplicative and asymptotically free in the sense that
  \begin{itemize}
      \item [(i)] For all $s,t \in G$, one has
      $$\lim_{n \rightarrow \infty}\frac{|\{v \in V_n:\sigma_{n}(st)v=\sigma_{n}(s)v\sigma_{n}(t)v\}|}{|V_{n}|}=1;$$
      \item [(ii)] for any distinct $s,t \in G$, one has $$\lim_{n \rightarrow \infty}\frac{|\{v \in V_n:\sigma_{n}(s)v=\sigma_{n}(t)v\}}{|V_n|}=0.$$
  \end{itemize}
  A group is \emph{sofic} if it admits a sofic approximation.
   In this paper, every $\sigma_n$ is a homomorphism, so only asymptotic freeness remains to be checked: 
	\begin{equation}\label{eq:sofic-free}
		\frac{|\Fix(\sigma_n(g))|}{|V_n|}\longrightarrow0
		\qquad(g\ne e),
	\end{equation}
	where $\Fix(\sigma_n(g)):=\{v\in V_n:\sigma_n(g)v=v\}$.
    \subsection{Sofic mean dimension and sofic metric mean dimension}
	Let $G\curvearrowright X$ be a continuous action on a compact metrizable space, and let $\rho$ be a continuous pseudometric on $X$.  For maps $\varphi,\psi:V\to X$, set
	\[
	\rho_2(\varphi,\psi)
	=\left(\frac1{|V|}\sum_{v\in V}\rho(\varphi(v),\psi(v))^2\right)^{\frac{1}{2}},
	\qquad
	\rho_\infty(\varphi,\psi)
	=\max_{v\in V}\rho(\varphi(v),\psi(v)).
	\]
	For a finite nonempty $F\subseteq G$, $\delta>0$, and $\sigma:G\to\Sym(V)$, define
	\[
	\Map(\rho,F,\delta,\sigma)
	=\left\{\varphi:V\to X:
	\rho_2(\varphi\circ\sigma(s),s\varphi)\leq \delta
	\text{ for all }s\in F\right\}.
	\]

For two finite open covers $\alpha$ and $\beta$ of $X$, we say that $\beta$ \emph{refines} $\alpha$, defined $\beta \succ \alpha$, if every member of $\beta$ is contained in a member of $\alpha$.
Denote
	\[
	\mathrm{ord}(\alpha) = \left( \max_{x \in X} \sum_{U \in \alpha} 1_U(x) \right) - 1\text{,~~}
	D(\alpha) = \min_{\beta \succ \alpha} \mathrm{ord}(\beta),
	\]
	where the minimum is taken over all finite open covers $\beta$ of $X$ such that $\beta$  refines $\alpha$.
    For a finite open cover $\cU$ of $X$, let $\cU^V$ be the product cover of $X^V$, and write
	\[
	D(\cU,\rho,F,\delta,\sigma)
	=D\left(\cU^V|_{\Map(\rho,F,\delta,\sigma)}\right).
	\]
	For a sofic approximation sequence $\Sigma$,  Li \cite{Li13} defined the \emph{sofic mean dimension} with respect to $\Sigma$ by
	\[
	\mdim_\Sigma(X)
	=\sup_{\cU}\inf_F\inf_{\delta>0}
	\limsup_{n\to\infty}
	\frac{D(\cU,\rho,F,\delta,\sigma_n)}{|V_n|},
	\]
	where $\cU$ ranges over finite open covers of $X$ and $F$ in the first infimum is over all finite sets $F \subseteq G$. The value is independent of the compatible metric on $X$.
    
	For a pseudometric space $(Y,\rho')$ and $\varepsilon>0$, a subset $E\subseteq Y$ is called $(\rho',\varepsilon)$-separated if for every distinct $x,y\in Y$,  $\rho'(x,y)>\varepsilon$. We write
$N_\varepsilon(Y,\rho')$
for the largest cardinality of an $(\rho',\varepsilon)$-separated subset of $Y$.

	 Put
	\[
	h^\varepsilon_{\Sigma,\infty}(\rho,F,\delta)
	=\limsup_{n\to\infty}\frac1{|V_n|}
	\log N_\varepsilon\bigl(\Map(\rho,F,\delta,\sigma_n),\rho_\infty\bigr),
	\]
    \[h^\varepsilon_{\Sigma,\infty}(\rho)=\inf_{F}\inf_{\delta>0}h^\varepsilon_{\Sigma,\infty}(\rho,F,\delta),
    \]
	where take the infimum over $\delta>0$ and all finite sets $F \subseteq G$.  The \emph{lower and upper sofic metric mean dimensions} with respect to $\Sigma$ and $\rho$ are
	\[
	\underline{\mdim}_{\Sigma,\mathrm M}(X,\rho)
	=\liminf_{\varepsilon\downarrow0}
	\frac{h^\varepsilon_{\Sigma,\infty}(\rho)}{|\log\varepsilon|},
	\qquad
	\overline{\mdim}_{\Sigma,\mathrm M}(X,\rho)
	=\limsup_{\varepsilon\downarrow0}
	\frac{h^\varepsilon_{\Sigma,\infty}(\rho)}{|\log\varepsilon|}.
	\]
	We also define
	\[
	\mdim_{\Sigma,\mathrm M}(X)
	=\inf_{\rho}\underline{\mdim}_{\Sigma,\mathrm M}(X,\rho),
	\]
	where $\rho$ ranges over all compatible metrics on $X$.
	
	A continuous pseudometric $\rho$ is dynamically generating if for every $x\ne y$ there exists $g\in G$ with $\rho(gx,gy)>0$.  Li proved that
	\begin{equation}\label{eq:metric-inf-dyn}
		\mdim_{\Sigma,\mathrm M}(X)
		=\inf_{\rho}\underline{\mdim}_{\Sigma,\mathrm M}(X,\rho),
	\end{equation}
	where the infimum on the right ranges over dynamically generating continuous pseudometrics on $X$, and also that
	\begin{equation}\label{eq:comparison}
		\mdim_\Sigma(X)\le \mdim_{\Sigma,\mathrm M}(X).
	\end{equation}
	See \cite[Lemma 4.4, Proposition 4.5 and Section 6]{Li13}.
    
    We use the following standard covering lemma.
	
	\begin{lemma}[Essential cube cover]\label{lem:cube-cover}
		Let
		\[
		\cV=\{[0,2/3),\ (1/3,1]\}
		\]
		be regarded as a relatively open cover of $[0,1]$.  Then
		\[
		D(\cV^{m})=m
		\qquad(m\ge1).
		\]
	\end{lemma}
	
	\begin{proof}
		The upper bound follows from $\dim([0,1]^m)=m$.  For the lower bound, every member of the product cover misses at least one of the two opposite faces
		\[
		F_i^0=\{x_i=0\},\qquad F_i^1=\{x_i=1\}
		\]
		for each $i$. By \cite[Lemma 3.2]{LW00} (see also \cite[Theorem IV.2]{HW}), we know that $D(\cV^m)\ge m$.
	\end{proof}
	
	\subsection{Action graphs and Cayley graphs}
	\begin{definition}[Action graph]
	Let $\sigma:F_r\to\Sym(V)$ be a homomorphism.  Define the undirected multigraph $\cG_\sigma$ by

\begin{itemize}
    \item[(1)] Vertex set: $V(\cG_\sigma)=V$;
    \item[(2)] Edge set: 
    \[
    E(\cG_\sigma)=\{\{v,\sigma(s_j)v\}:v\in V,\ 1\le j\le r\}.
    \]
\end{itemize}

	A set $I\subseteq V$ is \emph{independent} if it contains no two endpoints of an edge and contains no looped vertex.  Its maximum size is $\ind(\cG_\sigma)$.  Multiple edges do not affect independence.
    \end{definition}
	\begin{definition}[Cayley graph]
Let $G$ be a group and let $S \subseteq G$ be a generating set of $G$. The \emph{Cayley graph} of $G$ with respect to $S$ is the graph $\operatorname{Cay}(G, S)$ defined by:
\begin{itemize}
    \item[(1)] Vertex set: $G$;
    \item[(2)] Edge set: 
    \[
    \bigl\{\{g, g s\} : g \in G,\; s \in (S \cup S^{-1}) \setminus \{e\}\bigr\}.
    \]
\end{itemize}
\end{definition}
	
	\section{The properties of the hard-core system}\label{sec3}	\subsection{The hard-core system is topologically mixing}
	
	We now fix $G=F_r$ and the hard-core subshift $X_{\rm hc}$ from \eqref{eq:hardcore}.
	
	Recall that an action $G\curvearrowright X$ is topologically mixing if, for every pair of nonempty open sets $U,V\subseteq X$, there is a finite set $K\subseteq G$ such that $gU\cap V\ne\varnothing$ for every $g\notin K$.
	
	\begin{proposition}\label{prop:mixing}
		The action $F_r\curvearrowright X_{\rm hc}$ is topologically mixing.  More precisely, it has safe symbol $0$: any two legal finite patterns whose supports have graph distance at least two can be combined and extended by $0$ outside their domains.
	\end{proposition}
	
	\begin{proof}
		Let $p:A\to[0,1]$ and $q:B\to[0,1]$ be restrictions of points in $X_{\rm hc}$, where $A,B\subset F_r$ are finite.  Assume that $A\cap B=\varnothing$ and no Cayley edge has one endpoint in $A$ and the other in $B$.  Define $z$ by $z|_A=p$, $z|_B=q$, and $z_g=0$ elsewhere.  Every edge contained in $A$ or in $B$ satisfies the hard-core constraint because the corresponding pattern is legal.  Every other edge has at least one endpoint outside $A\cup B$, where $z$ is zero.  Hence $z\in X_{\rm hc}$. Next we show that the action $F_r\curvearrowright X_{\rm hc}$ is topologically mixing.
		
		For nonempty open sets $U,V$, choose nonempty cylinder sets $U'\subseteq U$ and $V'\subseteq V$ with finite defining domains $A$ and $B$.  Choose $x\in U'$ and $y\in V'$.  For a given $g\in F_r$, prescribe on $A$ the pattern $x|_A$ and on $g^{-1}B$ the translated pattern
		\[
		q_g(g^{-1}b)=y_b\qquad(b\in B).
		\]
		Both patterns are legal, since they are restrictions of $x$ and $g^{-1}y$, respectively.  Outside a finite set of $g\in F_r$, the sets $A$ and $g^{-1}B$ are disjoint and no Cayley edge joins them.  The preceding gluing then gives $z\in X_{\rm hc}$ satisfying $z\in U'$ and
		\[
		(gz)_b=z_{g^{-1}b}=y_b\qquad(b\in B),
		\]
		so $gz\in V'$.  Thus $gU\cap V\ne\varnothing$ for all $g$ outside a finite set. Then the action $F_r\curvearrowright X_{\rm hc}$ is topologically mixing.
	\end{proof}

	\subsection{Topological lower bound}
Let
	\[
	\rho_0(x,y)=|x_e-y_e|,\text{for any~}x=(x_{g})_{g \in F_r},y=(y_{g})_{g \in F_r}\in X_{hc}.
	\]
	Then $\rho_0$ is a dynamically generating pseudometric on $X_{hc}$, since shifts move every coordinate to the identity.
	The next construction is the source of all lower bounds.
	
	\begin{lemma}[Exact cube microstates]\label{lem:exact-cube}
		Let $\sigma:F_r\to\Sym(V)$ be a homomorphism and $I\subseteq V$ a independent set in $\cG_\sigma$.  For $z=(z_v)_{v\in I}\in[0,1]^I$, define $a_z:V\to[0,1]$ by
		\[
		a_z(v)=\begin{cases}z_v,&v\in I,\\0,&v\notin I,\end{cases}
		\]
		and define $\Phi_z:V\to X_{\rm hc}$ by
		\begin{equation}\label{eq:pullback}
			\Phi_z(v)_g=a_z(\sigma(g^{-1})v).
		\end{equation}
		Then:
		\begin{enumerate}[label=\textup{(\roman*)}]
			\item $\Phi_z(v)\in X_{\rm hc}$ for every $v\in V$;
			\item $\Phi_z(\sigma(h)v)=h\Phi_z(v)$ for every $h\in F_r$ and $v\in V$;
			\item $z\mapsto\Phi_z$ is a topological embedding of $[0,1]^I$ into every space $\Map(\rho,F,\delta,\sigma)$, for every continuous pseudometric $\rho$ on $X_{hc}$, finite $F\subseteq F_r$, and $\delta>0$;
			\item under $\rho_{0,\infty}$ this embedding is isometric for the supremum metric on $[0,1]^I$.
		\end{enumerate}
	\end{lemma}
	
	\begin{proof}
		For $g\in F_r$ and $s_j\in S$, put $u=\sigma(g^{-1})v$.  Since $\sigma$ is a homomorphism,
		\[
		\Phi_z(v)_g\Phi_z(v)_{gs_j}
		=a_z(u)a_z(\sigma(s_j^{-1})u)=0,
		\]
		because the two arguments are adjacent in $\cG_\sigma$ and the positive support of $a_z$ is contained in $I$.  This proves (i).
		
		For $h,g\in F_r$,
		\[
		\Phi_z(\sigma(h)v)_g
		=a_z(\sigma(g^{-1}h)v)
		=\Phi_z(v)_{h^{-1}g}
		=(h\Phi_z(v))_g,
		\]
		which proves (ii).  Thus $\Phi_z$ is exactly equivariant and belongs to every microstate space.  Continuity is immediate, and injectivity follows from
		\[
		\Phi_z(v)_e=a_z(v)=z_v\qquad(v\in I).
		\]
		Finally,
		\[
		\rho_{0,\infty}(\Phi_z,\Phi_{z'})
		=\max_{v\in V}|a_z(v)-a_{z'}(v)|
		=\max_{v\in I}|z_v-z'_v|.
		\]
	\end{proof}
	
	\begin{proposition}[Topological lower bound]\label{prop:top-lower}
		For every homomorphic sofic approximation $\Sigma$,
		\[
		\mdim_\Sigma(X_{\rm hc})\ge q_\Sigma.
		\]
	\end{proposition}
	
	\begin{proof}
 Let $\rho$ be a compatible metric on $X$, $F$ a finite subset of $F_r$ and $\delta>0$.
		Let
		\[
		U_0=\{x\in X_{\rm hc}:x_e<2/3\},
		\qquad
		U_1=\{x\in X_{\rm hc}:x_e>1/3\},
		\]
		and $\cU=\{U_0,U_1\}$. 
 For each $n$, take a maximum independent set $I_n$ in $\cG_{\sigma_n}$. 
 By Lemma \ref{lem:exact-cube}, for $z=(z_v)_{v\in {I_n}}\in[0,1]^{I_n}$, we can construst the exact cube microstates $\Phi_z(v)_g=a_z(\sigma(g^{-1})v)$, where $a_z:V_n\to[0,1]$ is
		\[
		a_z(v)=\begin{cases}z_v,&v\in {I_n},\\0,&v\notin {I_n}.\end{cases}
		\]
The map $\Theta:[0,1]^{I_n}\to \Map(\rho,F,\delta,\sigma_n)$, $\Theta(z)=\Phi_z$ is a topological embedding.
 Let
		\[
		\cV=\{[0,2/3),\ (1/3,1]\}
		\]
        Then by the construstion of $\Phi_z$, we have $\Theta^{-1}(\cU^{V_n}|_{\Theta([0,1]^{I_n})})=\cV^{I_{n}}$.
 Hence monotonicity under restriction and Lemma \ref{lem:cube-cover} give
		\begin{align*}
		 	D(\cU,\rho,F,\delta,\sigma_n) &\geq D(\Theta^{-1}(\cU^{V_n}|_{\Map(\rho,F,\delta,\sigma_n)}))\\
 &\geq D(\Theta^{-1}(\cU^{V_n}|_{\Theta([0,1]^{I_n})}))\\
 &=D(\cV^{I_{n}})\geq 
		 |I_n|=\ind(\cG_{\sigma_n})   
		\end{align*}
for  every finite $F$, and every $\delta>0$.  Taking the normalized limsup, then the infima over $F$ and $\delta$, proves the result.
	\end{proof}
	
	\begin{corollary}[Metric lower bound]\label{prop:metric-lower}
		For every homomorphic sofic approximation $\Sigma$,
		\[
		\overline{\mdim}_{\Sigma,\mathrm M}(X_{\rm hc},\rho_0)\ge \underline{\mdim}_{\Sigma,\mathrm M}(X_{\rm hc},\rho_0)\ge q_\Sigma.
		\]
	\end{corollary}
	
	\begin{proof}
	Combining \cite[Theorem 6.1]{Li13} and Proposition \ref{prop:top-lower}, we have
    \[ q_\Sigma \leq \mdim_\Sigma(X_{\rm hc}) \leq \underline{\mdim}_{\Sigma,\mathrm M}(X_{\rm hc},\rho_0) \leq \overline{\mdim}_{\Sigma,\mathrm M}(X_{\rm hc},\rho_0).
    \]
      This completes the proof.
	\end{proof}
	
	\subsection{The metric upper bound}
	
	The key point is that a microstate can have many coordinates above a fixed threshold only where the action graph is almost independent.
	
	\begin{lemma}\label{lem:threshold}
		Let $\sigma:F_r\to\Sym(V)$ be a homomorphism and $F$ a finite subset of $F_r$ and $\delta>0$. Let
		\[
		F_0=\{s_1^{-1},\ldots,s_r^{-1}\},
		\]
		and let $\varphi\in\Map(\rho_0,F_0,\delta,\sigma)$.  Write
		\[
		y_v=\varphi(v)_e
		\qquad(v\in V)
		\]
		and, for $\tau>0$, set
		\[
		A_\tau=\{v\in V:y_v\ge\tau\}.
		\]
		Then one can delete at most $r\delta^2\tau^{-2}|V|$ vertices from $A_\tau$ and obtain an independent set.  In particular,
		\begin{equation}\label{eq:threshold-size}
			|A_\tau|
			\le \ind(\cG_\sigma)+r\delta^2\tau^{-2}|V|.
		\end{equation}
	\end{lemma}
	
	\begin{proof}
		
        For $1\le j\le r$, the microstate condition for $s_j^{-1}$ gives
		\begin{align*}
			\delta^2
			&\geq \rho_{0,2}(\varphi\circ\sigma(s_j^{-1}),s_j^{-1}\varphi)^2\\
			&=\frac1{|V|}\sum_{v\in V}
			\left|y_{\sigma(s_j^{-1})v}-\varphi(v)_{s_j}\right|^2.
		\end{align*}
		If both $v$ and $\sigma(s_j^{-1})v$ lie in $A_\tau$, then $y_v\ge\tau>0$.  Since $\varphi(v)\in X_{\rm hc}$,
		\[
		y_v\varphi(v)_{s_j}=0,
		\]
		so $\varphi(v)_{s_j}=0$.  The corresponding summand is therefore at least $\tau^2$.  Let
 \[ T_{j}=\{v\in V:v,\sigma(s_j^{-1})v\in A_\tau\},\]
 for any $1\leq j \leq r$.   Consequently
		\[
		\left|T_j\right|
    \le \delta^2\tau^{-2}|V|, \text{~for any ~} 1 \leq j \leq r.
		\]
		
Then $A_{\tau} \setminus \bigcup_{j=1}^{r}T_j$
is a indepedent set. Indeed, if an undirected generator edge remained inside $A_{\tau} \setminus \bigcup_{j=1}^{r}T_j$, then for a suitable orientation it would have been deleted above, a contradiction.  Loops are handled in the same way. Hence
        \[|A_{\tau}|-r \delta^2\tau^{-2}|V| \leq \left|A_{\tau} \setminus \bigcup_{j=1}^{r}T_j\right|\leq \ind(\cG_\sigma).\]
 This proves \eqref{eq:threshold-size}.
	\end{proof}
	
	\begin{proposition}[Metric upper bound]\label{prop:metric-upper}
		For every homomorphic sofic approximation $\Sigma$,
		\[
		\underline{\mdim}_{\Sigma,\mathrm M}(X_{\rm hc},\rho_0)\le 
		\overline{\mdim}_{\Sigma,\mathrm M}(X_{\rm hc},\rho_0)\le q_\Sigma.
		\]
	\end{proposition}
	
	\begin{proof}
     Let $F_0=\{s_1^{-1},\ldots,s_r^{-1}\}$  and $\delta>0$.
		Fix $0<\varepsilon<1/2$ and put $\tau=\varepsilon/4$.  For any $n\in\mathbb{N}$, let $E_n$ be an $(\rho_{0,\infty},\varepsilon)$-separated subset of
		\[
		\Map(\rho_0,F_0,\delta,\sigma_n).
		\]
		 For $\varphi\in E_n$, record the following data:
		\begin{enumerate}[label=\textup{(\roman*)}]
			\item the set $A_\tau(\varphi)=\{v:\varphi(v)_e\ge\tau\}$;
			\item for every $v\in A_\tau(\varphi)$, the interval of the partition of $[0,1]$ into intervals of length at most $\varepsilon/2$ which contains $\varphi(v)_e$.
		\end{enumerate}
		If two microstates have the same record, then their identity coordinates differ by less than $\varepsilon/2$ on the recorded set.  Outside it, both coordinates are less than $\tau$, so their difference is less than $\tau$.  Hence their $\rho_{0,\infty}$-distance is less than $\varepsilon$, contradicting separation.  Therefore the record map is injective on $E_n$.
		
		There are at most $2^{|V_n|}$ choices for $\bigcup_{\varphi \in E_n}A_{\tau}(\varphi)$ and at most $1+2/\varepsilon$ interval labels per active coordinate.  By Lemma \ref{lem:threshold},
		\[
		|E_n| \le 2^{|V_n|}
		\left(1+\frac2\varepsilon\right)^{
			\max_{\varphi \in E_n}|A_{\tau}(\varphi)|}
		\le 2^{|V_n|}
		\left(1+\frac2\varepsilon\right)^{
			\ind(\cG_{\sigma_n})+16r\delta^2\varepsilon^{-2}|V_n|}.
		\]
		It follows that
		\begin{align*}
			h^\varepsilon_{\Sigma,\infty}(\rho_0,F_0,\delta)
			&\le \log2+\left(q_\Sigma+16r\delta^2\varepsilon^{-2}\right)
			\log\left(1+\frac2\varepsilon\right).
		\end{align*}
		For fixed $\varepsilon$, take the infimum over $\delta>0$.  Since the definition also takes an infimum over finite $F$, this specific $F_0$ is naturally included.  Thus
		\[
		h^\varepsilon_{\Sigma,\infty}(\rho_0)
		\le \log2+q_\Sigma\log\left(1+\frac2\varepsilon\right).
		\]
		After division by $|\log\varepsilon|$, the constant $\log2$ disappears and the logarithmic ratio tends to $1$ as $\varepsilon\downarrow0$.  The same estimate therefore controls both the lower and upper limits, proving the proposition.
	\end{proof}
	
	\begin{proof}[Proof of Theorem \ref{thm:model-formula}]
		Propositions \ref{prop:metric-lower} and \ref{prop:metric-upper} give
		\[
		\underline{\mdim}_{\Sigma,\mathrm M}(X_{\rm hc},\rho_0)
		=\overline{\mdim}_{\Sigma,\mathrm M}(X_{\rm hc},\rho_0)
		=q_\Sigma.
		\]
		Since $\rho_0$ is dynamically generating, \eqref{eq:metric-inf-dyn} and the comparison inequality \eqref{eq:comparison} give
		\[
		\mdim_\Sigma(X_{\rm hc})
		\le \mdim_{\Sigma,\mathrm M}(X_{\rm hc})
		\le \underline{\mdim}_{\Sigma,\mathrm M}(X_{\rm hc},\rho_0)
		=q_\Sigma.
		\]
		The reverse inequality is Proposition \ref{prop:top-lower}.
	\end{proof}
	
	\section{A bipartite sofic approximation}\label{sec4}
	
	We specialize to
	\[
	F_2=\langle a,b\rangle.
	\]
	Let
	\[
	\chi:F_2\to\Z/2\Z,
	\qquad \chi(a)=\chi(b)=1,
	\]
	be the parity homomorphism.
	
	\begin{lemma}\label{lem:residual-chain}
		There is a sequence of finite-index normal subgroups $N_n\triangleleft F_2$ such that
		\[
		N_n\subseteq\ker\chi,
		\qquad
		\bigcap_{n\ge1}N_n=\{e\}.
		\]
		The left translation actions on $F_2/N_n$ form a homomorphic sofic approximation.
	\end{lemma}
	
	\begin{proof}
		Enumerate $F_2\setminus\{e\}$ as $g_1,g_2,\ldots$.  By residual finiteness, for every $m$ there is a finite-index normal subgroup $H_m\triangleleft F_2$ such that $g_m\notin H_m$.  Set
		\[
		K_n=\bigcap_{m=1}^n H_m,
		\qquad
		N_n=K_n\cap\ker\chi.
		\]
		Then $(N_n)$ is a decreasing sequence of finite-index normal subgroups, every $N_n$ is contained in $\ker\chi$, and
		\[
		\bigcap_{n\ge1}N_n=\{e\}.
		\]
		Moreover, if $g=g_m\ne e$, then $g\notin N_n$ for every $n\ge m$.  Hence left translation by $g$ on $F_2/N_n$ has no fixed point for all $n\ge m$: indeed, a fixed coset would imply $x^{-1}gx\in N_n$, and normality would then give $g\in N_n$.  Thus the quotient actions are asymptotically free.  Since the chain is decreasing with trivial intersection and $F_2$ is infinite, the indices $[F_2:N_n]$ tend to infinity.  Therefore the quotient actions form a homomorphic sofic approximation.
	\end{proof}
	
	Let $V_n=F_2/N_n$ and let $\sigma_n^{\rm bip}: F_2 \rightarrow \Sym{(V_n)}$ 
 be the quotient action.  
	Let
	\[
	\widetilde\chi:F_2/N_{n}\to\Z/2\Z,
	\qquad \widetilde\chi(gN_{n})=\chi(g).
	\]
Let
$L_n=\widetilde\chi^{-1}(0)$, and $R_n=\widetilde\chi^{-1}(1)$.
Since $N_n\subseteq\ker\chi$, parity descends to $V_n$ and gives a decomposition
	\[
	V_n=L_n\sqcup R_n,
	\qquad |L_n|=|R_n|=|V_n|/2.
	\]
	Both $a$ and $b$ exchange the two classes.
	
	\begin{proposition}\label{prop:bip-alpha}
		For every $n$,
		\[
		\ind(\cG_{\sigma_n^{\rm bip}})=\frac{|V_n|}{2}.
		\]
		Consequently
		\[
		\mdim_{\Sigma_{\rm bip}}(X_{\rm hc})=\frac12.
		\]
	\end{proposition}
	
	\begin{proof}
		Every generator edge crosses from $L_n$ to $R_n$, so $L_n$ is independent and
		\[
		\ind(\cG_{\sigma_n^{\rm bip}})\ge |V_n|/2.
		\]
		The map $v\mapsto av$ is a bijection from $L_n$ onto $R_n$.  Hence the $a$-edges indexed by $L_n$ form a perfect matching.  An independent set contains at most one endpoint from each matching edge, so its size is at most $|V_n|/2$.  The formula for mean dimension follows from Theorem \ref{thm:model-formula}.
	\end{proof}
	
	\section{A uniform random sofic approximation}\label{sec5}
	
	Let $d\in\N$ and choose $\pi_a,\pi_b\in\Sym(d)$ independently and uniformly.  Since $F_2$ is free, there is a unique homomorphism
	\[
	\sigma:F_2\to\Sym(d),
	\qquad \sigma(a)=\pi_a,
	\quad \sigma(b)=\pi_b.
	\]
    Let
\[
 H(t)=-t\log t-(1-t)\log(1-t),\qquad 0<t<1,
\]
be the binary entropy function.
	We first estimate independent sets in $\cG_\sigma$.
	
	\begin{lemma}\label{lem:first-moment}
		Fix $0<q<1/2$, and let $k_d=\lfloor qd\rfloor$.  Let $Z_{d,q}$ be the number of independent subsets of $\cG_\sigma$ of cardinality $k_d$.  Then
		\[
		\lim_{d\to\infty}\frac1d\log\mathbb E Z_{d,q}
		=\Phi_2(q),
		\]
		where
		\[
		\Phi_2(q)=H(q)+2\bigl(2(1-q)\log(1-q)
		-(1-2q)\log(1-2q)\bigr).
		\]
		In particular, if $\Phi_2(q)<0$, when $d \longrightarrow \infty$, 
		\[
		\mathbb P\bigl(\ind(\cG_\sigma)\ge \lfloor qd\rfloor\bigr)\longrightarrow0
		\]
		exponentially fast.  Consequently, when $d \longrightarrow \infty$, 
		\[
		\mathbb P\bigl(\ind(\cG_\sigma)\ge qd\bigr)\longrightarrow0.
		\]
	\end{lemma}
	
	\begin{proof}
		Firstly, suppose $qd$ is integral and fix $I\subseteq[d]$ with $|I|=qd$.  If $I$ is an independent subsets of $\cG_\sigma$, then the condition that $I$ contain no edge arising from a permutation $\pi$ is
		\[
		\pi(I)\cap I=\varnothing.
		\]
		The image $\pi(I)$ is uniformly distributed among the $qd$-subsets of $[d]$, so
		\begin{align*}
			\mathbb P(\pi(I)\cap I=\varnothing)
			&=\frac{\binom{(1-q)d}{qd}}{\binom d{qd}}\\
			&=\frac{((1-q)d)!^2}{((1-2q)d)!d!}.
		\end{align*}
		The two generator permutations are independent.  Therefore
		\[
		\mathbb E Z_{d,q}
		=\binom d{qd}
		\left(\frac{((1-q)d)!^2}{((1-2q)d)!d!}\right)^2.
		\]
	By	Stirling's formula, we have
  \[\mathbb E Z_{d,q}
		=\exp(d(\Phi_2(q)+o(1))).
		\]
  Replacing $qd$ by $\lfloor qd\rfloor$ changes the normalized logarithm by $o(1)$.
		
		If an independent set has size at least $k_d$, it contains an independent subset of size exactly $k_d$. Note that $\{\ind(\cG_\sigma)\ge k_d\}\subseteq\{Z_{d,q}\leq 1\}$.
        Hence by Markov's inequality, we have
		\[
		\mathbb P(\ind(\cG_\sigma)\ge k_d)
		\le \mathbb E Z_{d,q}.
		\]
		When $\Phi_2(q)<0$, the right side decays exponentially.  Finally,
		\[
		\{\ind(\cG_\sigma)\ge qd\}
		\subseteq
		\{\ind(\cG_\sigma)\ge \lfloor qd\rfloor\},
		\]
		which proves the last assertion.
	\end{proof}
	
	\begin{lemma}\label{lem:phi-negative}
		One has
		\[
		\Phi_2(9/20)<-0.16.
		\]
		In particular, the conclusion of Lemma \ref{lem:first-moment} applies with $q=9/20$.
	\end{lemma}
	
	\begin{proof}
		Substitution gives
		\[
		\Phi_2(9/20)
		=H(9/20)+2\left(\frac{11}{10}\log\frac{11}{20}
		-\frac1{10}\log\frac1{10}\right)
		=-0.166585\ldots<-0.16.
		\]
	\end{proof}
	
	We also need the standard asymptotic freeness of the random permutation model.
	
	\begin{lemma}[Random permutations are asymptotically free]\label{lem:random-free}
		For every nontrivial reduced word $w\in F_2$ there is a constant $C_w<\infty$ such that
		\[
		\mathbb E|\Fix(\sigma(w))|\le C_w
		\]
		for all sufficiently large $d$.  Consequently, when $d\longrightarrow\infty$,
		\[
		\frac{|\Fix(\sigma(w))|}{d}\longrightarrow0
		\]
		in probability.
	\end{lemma}
	
	\begin{proof}
Write the reduced word as
\[
w=t_\ell\cdots t_1,
\qquad
 t_i\in\{a,a^{-1},b,b^{-1}\}.
\]
Fix a vertex $v\in[d]$ and follow the path described by the word $w$:
\[
x_0=v,
\qquad
x_i=\sigma(t_i)x_{i-1}
\quad(1\leq i\leq\ell).
\]
Then
\[
x_\ell=\sigma(w)v.
\]
Thus $v$ is fixed by $\sigma(w)$ exactly when $x_\ell=x_0$.

We reveal the values of $\pi_a$ and $\pi_b$ only when they are needed along
this path. Suppose that, before step $i$, the vertices
\[
x_0,x_1,\ldots,x_{i-1}
\]
are all distinct. Since $w$ is reduced, step $i$ cannot simply reverse the
edge used at step $i-1$. Hence the required value of $\pi_a$, $\pi_a^{-1}$,
$\pi_b$, or $\pi_b^{-1}$ has not yet been determined.

At most $i-1$ values of the relevant permutation have already been exposed.
Therefore the next vertex $x_i$ is uniformly distributed among at least
\[
d-(i-1)\geq d-\ell
\]
possible vertices. Among these possibilities, at most the $i$ previously
visited vertices
\[
x_0,\ldots,x_{i-1}
\]
produce a collision. Hence
\[
\mathbb P\bigl(\text{the first collision occurs at step }i\bigr)
\leq \frac{i}{d-\ell}.
\]

If $x_\ell=x_0$, then a collision must occur at some step
$1\leq i\leq\ell$. By the union bound,
\[
\begin{aligned}
\mathbb P\bigl(\sigma(w)v=v\bigr)
&\leq
\sum_{i=1}^{\ell}\frac{i}{d-\ell}\\
&=
\frac{\ell(\ell+1)}{2(d-\ell)}.
\end{aligned}
\]
For $d>2\ell$, this gives
\[
\mathbb P\bigl(\sigma(w)v=v\bigr)
\leq \frac{\ell(\ell+1)}{d}.
\]

Now sum over all $v\in[d]$. Since
\[
|\Fix(\sigma(w))|
=
\sum_{v=1}^{d}\mathbf 1_{\{\sigma(w)v=v\}},
\]
linearity of expectation gives
\[
\begin{aligned}
\mathbb E|\Fix(\sigma(w))|
&=
\sum_{v=1}^{d}\mathbb P\bigl(\sigma(w)v=v\bigr)\\
&\leq
 d\cdot\frac{\ell(\ell+1)}{d}\\
&=
\ell(\ell+1).
\end{aligned}
\]
Thus we may take $C_w=\ell(\ell+1)$.

Finally, for every $\eta>0$, by Markov's inequality, when $d \longrightarrow \infty$, 
\[
\mathbb P\left(
\frac{|\Fix(\sigma(w))|}{d}>\eta
\right)
\leq
\frac{\mathbb E|\Fix(\sigma(w))|}{\eta d}
\leq
\frac{\ell(\ell+1)}{\eta d}
\longrightarrow0.
\]
This proves the claimed convergence in probability.
\end{proof}
	
	\begin{proposition}[Deterministic uniform approximation]\label{prop:uniform-sequence}
		There exists a homomorphic sofic approximation
		\[
		\Sigma_{\rm unif}=(\sigma_n^{\rm unif}:F_2\to\Sym(V_n))_{n \geq 1}
		\]
		such that
		\[
		\limsup_{n\to\infty}
		\frac{\ind(\cG_{\sigma_n^{\rm unif}})}{|V_n|}
		\le\frac9{20}.
		\]
		Moreover,
		\[
		\liminf_{n\to\infty}
		\frac{\ind(\cG_{\sigma_n^{\rm unif}})}{|V_n|}
		\ge\frac15.
		\]
	\end{proposition}
	
	\begin{proof}
		Enumerate the nontrivial elements of $F_2$ as $w_1,w_2,\ldots$.  For each $n$, take $d$ sufficiently large.  By Lemma \ref{lem:first-moment} and Lemma \ref{lem:phi-negative}, with probability tending to one,
		\[
		\ind(\cG_\sigma)<\frac9{20}d.
		\]
		By Lemma \ref{lem:random-free} and a union bound over the finitely many words $w_1,\ldots,w_n$, also with probability tending to one,
		\[
		|\Fix(\sigma(w_j))|<\frac dn
		\qquad(1\le j\le n).
		\]
		For large enough $d$, the intersection of these events with
		\[
		\ind(\cG_\sigma)<\frac9{20}d
		\]
		is nonempty.  Choose one realization and call it $\sigma_n^{\rm unif}$, with $d=|V_n|$ chosen strictly increasing.  The fixed-point estimates imply \eqref{eq:sofic-free}, so the sequence is a homomorphic sofic approximation.  The upper independence bound is built into the choice.
		
		For the lower bound, let $L_n$ be the set of vertices carrying a loop in $\cG_{\sigma_n^{\rm unif}}$.  Such a loop can arise only from a fixed point of $\sigma_n^{\rm unif}(a)$ or $\sigma_n^{\rm unif}(b)$, so $|L_n|=o(|V_n|)$ by soficity.  After deleting $L_n$, the underlying simple graph has maximum degree at most $4$.  The greedy algorithm gives an independent set of size at least
		\[
		\frac{|V_n|-|L_n|}{4+1}.
		\]
		Dividing by $|V_n|$ proves the liminf estimate.
	\end{proof}
	
	\begin{proof}[Proof of Theorem \ref{thm:main}]
		For $\Sigma_{\rm bip}$, Proposition \ref{prop:bip-alpha} and Theorem \ref{thm:model-formula} give
		\[
		\mdim_{\Sigma_{\rm bip}}(X_{\rm hc})=\frac12.
		\]
		For $\Sigma_{\rm unif}$, Proposition \ref{prop:uniform-sequence} and Theorem \ref{thm:model-formula} give
		\[
		\frac15\le\mdim_{\Sigma_{\rm unif}}(X_{\rm hc})\le\frac9{20}.
		\]
		The strict inequality $9/20<1/2$ completes the proof.  Topological mixing was established in Proposition \ref{prop:mixing}.
	\end{proof}

\end{document}